\documentclass[12pt]{article}

\usepackage{amsthm,amsmath}
\RequirePackage[numbers]{natbib}
\RequirePackage[colorlinks,citecolor=blue,urlcolor=blue]{hyperref}
\usepackage{amsfonts}
\usepackage{graphicx}
\usepackage{amsthm}
\usepackage{mathtools}
\usepackage{amsmath}
\usepackage{caption}
\usepackage{subcaption}
\usepackage{pgfplotstable}
\usepackage{blindtext}
\usepackage{float}
\usepackage[utf8]{inputenc}
\usepackage{booktabs}
\usepackage{amssymb}
\usepackage{fullpage}

\DeclareMathOperator{\hsic}{HSIC}
\DeclareMathOperator{\mmd}{MMD}

\DeclareMathOperator{\id}{id}

\newcommand\independent{\protect\mathpalette{\protect\independenT}{\perp}}
\def\independenT#1#2{\mathrel{\rlap{$#1#2$}\mkern2mu{#1#2}}}

\newtheorem{theorem}{Theorem}
\newtheorem{definition}{Definition}
\newtheorem{setting}{Setting}

\begin{document}


\title{Consistency of permutation tests for HSIC and dHSIC}
\author{David Rindt, Dino Sejdinovic and David Steinsaltz}
\date{%
    Department of Statistics,\\
     University of Oxford}
    
\maketitle



\begin{center}
\section*{Abstract}
\end{center}
The Hilbert--Schmidt Independence Criterion (HSIC) is a popular measure of the dependency between two random variables. The statistic dHSIC is an extension of HSIC that can be used to test joint independence of $d$ random variables. Such hypothesis testing for (joint) independence is often done using a permutation test, which compares the observed data with randomly permuted datasets. The main contribution of this work is proving that the power of such independence tests converges to 1 as the sample size converges to infinity. This answers a question that was asked in \cite{pfister2018kernel}. Additionally this work proves correct type 1 error rate of HSIC and dHSIC
permutation tests and provides guidance on how to select the number of
permutations one uses in practice. While correct type 1 error rate was already proved in \cite{pfister2018kernel}, we provide a modified proof following \cite{berrettinformation}, which extends to the case of non-continuous data. The number of permutations to use was studied e.g. by \cite{marozzi2004some} but not in the context of HSIC and with a slight difference in the estimate of the $p$-value and for permutations rather than vectors of permutations. While the last two points have limited novelty we include these to give a complete overview of permutation testing in the context of HSIC and dHSIC.

\section{Introduction} 

In \cite{kernelteststatisticalindependence} and \cite{kerneltwosampletest} kernel methods were proposed for independence testing and two-sample testing. Since then kernel based tests have been proposed for conditional independence testing and joint independence testing \cite{zhang2012kernel},\cite{pfister2018kernel}. Such tests have been used in graphical modeling, among other applications. Independence testing using reproducing kernel Hilbert spaces has also been extended to right-censored data found in survival analysis \cite{rindt2019nonparametric,fernandez2019kernel}. We study the tests for joint independence proposed by \cite{pfister2018kernel} which includes the independence test between two random variables.\\

These methods have several desirable properties. For appropriate choices of kernel, the population value of the test statistic, called the Hilbert--Schmidt Independence Criterion (HSIC), equals zero if and only if the two variables are independent. Similarly, the population value of the statistic measuring joint independence --- the $d$-variable HSIC, or {\em dHSIC} --- is zero if and only if 
the variables are indeed jointly independent. One thus does not need to make assumptions about the form of the relationship among 
the variables. Furthermore, under mild conditions the test statistic converges in probability to the population value. Additionally, these tests may be applied to multidimensional random variables, and even to variables 
that do not take values in the Euclidean domains, such as graphs or text \cite{kernelteststatisticalindependence}.\\

In practice, one does not have access to the true sampling distribution. To perform hypothesis testing one thus needs to approximate the null distribution or perform permutation tests or bootstrapping. These three methods were studied for dHSIC in \cite{pfister2018kernel} by Pfister, B{\"u}hlmann, Sch{\"o}lkopf, and Peters, where they established consistency of the bootstrap test (power converging to 1 for every alternative hypothesis), correct type 1 error rate of the permutation test, and pointwise asymptotic correct type 1 error rate of the bootstrap procedure. \\

One question that remained unanswered was the consistency of the permutation test. See Table 1 of \cite{pfister2018kernel} and Section 3.2.1 and Remark 2 where they propose a proof strategy. The main theoretical contribution of this work is to prove the consistency of the permutation test, albeit not in the proposed way, but using more elementary techniques that can be traced back at least to \cite{hoeffding1952large}: as we discuss in Section \ref{sec:introduction} the test statistic dHSIC, with appropriate choice of kernel, converges to a positive constant for each fixed alternative hypothesis. The main observation from which consistency will follow 
is that it suffices for the statistic's distribution under random permutation of the data to converge to zero in probability (Theorem \ref{thm:permutedhsic}). The full proof of consistency may be found in Section \ref{sec:consistency}. \\

We also present short proofs the permutation test has correct type 1 error rate (Section \ref{sec:level}) and investigate the question of how many permutations are appropriate to use (Section \ref{sec:howmany}). These last investigations are not new, and can be found 
elsewhere in the literature, e.g. \cite{pfister2018kernel,marozzi2004some,berrettinformation}, as well as older literature, such as \cite{hoeffding1952large}. We review these ideas here for completeness, and because we wish to give a more unified treatment. 
In particular, \cite{marozzi2004some} studied the number of permutations, but differs from our notation in considering
individual permutations rather than vectors of permutations, and their $p$-value estimate lacked a guarantee for the type 1 error rate of the test. In \cite{pfister2018kernel} correct type 1 error rate of the test was proved, but under the additional assumption that the random variables had a density. For completeness we also show here two correct ways of dealing with non-continuous data. Furthermore, we provide a different proof, following \cite{berrettinformation}, which appeared in the context of independence-testing using mutual information.


\section{Background}\label{sec:introduction}

\subsection{Reproducing Kernel Hilbert Spaces}
This section reviews some relevant information about reproducing kernel Hilbert spaces (RKHSs).

\begin{definition}{ (Reproducing Kernel Hilbert Space)}(\cite{learningwithkernels}) Let $\mathcal X$ be a non-empty set and $H$ a Hilbert space of functions $f:\mathcal X \to \mathbb R$. Then $H$ is called a reproducing kernel Hilbert (RKHS) space endowed with dot product $\langle \cdot,\cdot \rangle $ if there exists a function $k: \mathcal X \times \mathcal X \to \mathbb R$ with the following properties.
\begin{enumerate}
\item $k$ has the reproducing property 
\begin{equation}
\langle f,k(x,\cdot)\rangle=f(x) \ \text{for all} \ f \in H, x \in \mathcal X.
\end{equation}
\item $k$ spans $H$, that is, $H= \overline{\mathrm{span}\{k(x,\cdot) \ \vert x \in \mathcal X \}}$ where the bar denotes the completion of the space.
\end{enumerate}  \end{definition}
Let $\mathcal X$ be a measurable space and $H_{k}$ be an RKHS on $\mathcal X$ with kernel $k$. Let $\mathbb P$ be a probability measure on $\mathcal X$.  If $\mathbb E_{\mathbb P}\sqrt{k(X,X)}<\infty$, then there exists an element $\mu_{\mathbb P}\in H_{\mathcal X}$ such that $\mathbb E_{\mathbb P}f(X)=\langle f,\mu_{\mathbb P} \rangle$ for all $f\in H_{\mathcal X}$ (\cite{kerneltwosampletest}), where we use the notation $\mathbb E_{\mathbb P} f(X)\coloneqq \int f(x) \mathbb{P}(dx)$. The element $\mu_{\mathbb P}$ is called the mean embedding of $\mathbb P$ in $H_{k}$. Given a sample $\{x_i\}_{i=1}^n$ and the corresponding empirical distribution,  $\frac1n \sum_{i=1}^n \delta_{x_i}$, the corresponding mean embedding is given by $
\frac 1n \sum_{i=1}^n k(x_i,\cdot).$  Given a second distribution $\mathbb Q$ on $\mathcal X$, of which a mean embedding exists, we can measure the dissimilarity of $\mathbb P$ and $\mathbb Q$ by the distance between their mean embeddings in $H_{\mathcal X}$. That is,
\begin{align*}
\mmd(\mathbb P,\mathbb Q)\coloneqq \vert \vert \mu_{ \mathbb P} - \mu_{ \mathbb Q} \vert \vert _{H_k}.
\end{align*}
This is also called the Maximum Mean Discrepancy ($\mmd$). The name comes from the following equality \cite{kerneltwosampletest}, 
\begin{align*}
\vert \vert \mu_{ \mathbb P} - \mu_{ \mathbb Q} \vert \vert _{H_k} = \sup_{f \in H_k} \mathbb E_{ \mathbb P}  f(X) - \mathbb E_{ \mathbb Q}f(X),
\end{align*}
showing that MMD is an integral probability metric. The kernel $k$ is said to be characteristic when $\mmd(\mathbb P,\mathbb Q)=0$ if and only if $\mathbb P=\mathbb Q$. Lastly, for a locally compact Hausdorff space $\mathcal X$, the kernel $k$ is said to be $c_0$-universal if it is continuous and $H_k$ is dense in  $C_0({\mathcal X})$, the set of continuous bounded functions, with respect to the infinity (also called uniform) norm \cite{sriperumbudur2011universality}. The most commonly used example of a kernel that is both characteristic and $c_0$-universal is the Gaussian kernel $k_{\sigma}(x,y)=\exp(- \vert \vert x-y \vert \vert^2/\sigma^2)$ on $\mathbb R^d$.

\subsection{dHSIC}

In \cite{pfister2018kernel} Pfister, B{\"u}hlmann, Sch{\"o}lkopf, and Peters propose a kernel based test for joint independence. Consider the following setting.\\

\begin{setting} For $i=1,\dots,d$, let $\mathcal X^i$ be a locally compact metric space equipped with the Borel sigma-algebra. Let $\mathcal X=\mathcal X^1\times \cdots \times \mathcal X^d$ be equipped with the product sigma-algebra. Let $X^i:\Omega\to\mathcal X^i$ be a random variables on the shared probability space $(\Omega,\mathbb P,\mathcal F)$. In this section the superscript on
$X^j$ always indexes $X$, and never denotes a power of the variable $X$. Let $k^i(\cdot,\cdot):\mathcal X^i \times \mathcal X^i\to \mathbb R$ be a $c_0$-universal kernel on $\mathcal X^i$. Finally let $k\coloneqq k^1 \otimes \cdots \otimes k^d$ be the tensor product of the $d$ kernels. By \cite{szabo2017characteristic}, $k$ is characteristic and $c_0$-universal on $\mathcal X$. We let $\mathcal H_{k},\mathcal H_{k^i}$ be the corresponding RKHSs. \end{setting}

By definition $(X^1,...,X^d)$ are said to be jointly independent if 
$\mathbb P_{X^1,...,X^d}=\mathbb P_{X^1}\times \cdots \mathbb \times \mathbb P_{X^d}$. The main topic of this work is the hypothesis test where $H_0:\mathbb P_{X^1,...,X^d}=\mathbb P_{X^1}\times \cdots \mathbb \times \mathbb P_{X^d}$. With this in mind, we define dHSIC,
\begin{definition}(dHSIC \cite{pfister2018kernel}) Assume Setting 1. Then dHSIC is defined as
\begin{align*}
\text{dHSIC}(X^1,...,X^d)\coloneqq \lvert \lvert \mu_{\mathbb P_{X^1,...,X^d}} - \mu_{\mathbb P_{X^1}\times \cdots \mathbb \times \mathbb P_{X^d}} \rvert \rvert_{\mathcal H_k} ^2.
\end{align*}
\end{definition}
\noindent
Note that, because $k$ is characteristic, $\text{dHSIC}(X^1,\dots,X^d)= 0$ if and only if $\mathbb P_{X^1,...,X^d}=\mathbb P_{X^1}\times \cdots \mathbb \times \mathbb P_{X^d}$. 

As we typically do not have access to the full distribution $\mathbb P_X$, but only to a sample $D\coloneqq (x_i)_{i=1}^n\in \mathcal X^n$, we study the estimator \cite{pfister2018kernel}
\begin{align*}
\widehat{ \text{dHSIC}}(x_1,\dots,x_n) \coloneqq & \frac{1}{n^2} \sum_{M_2(n)} \prod_{j=1}^dk^j(x_{i_1}^j,x^j_{i_2})+\frac{1}{n^{2d}}\sum_{M_{2d}(n)}\prod_{j=1}^d k^j(x^j_{i_{2j-1}},x^j_{i_{2j}})\\
&-\frac{2}{n^{d+1}}\sum_{M_{d+1}(n)}\prod_{j=1}^dk^j(x^j_{i_1},x^j_{i_{j+1}}).
\end{align*}
Here $M_q(n)=\{1,\dots,n\}^q$. Note that this equals the RKHS distance between the mean embedding of the empirical distribution and of the product distribution. A final important property is that,
\begin{align*}
\widehat{ \text{dHSIC}}(x_1,\dots,x_n) \to \text{dHSIC}(X^1,...,X^d)
\end{align*} 
as $n\to \infty$ in probability \cite{pfister2018kernel}.
\subsection{HSIC}
In the case where $d=2$, dHSIC coincides with HSIC, which is defined as 
\begin{definition}
The Hilbert--Schmidt independence criterion (HSIC) of random variables $X^1\in \mathcal X^1$ and $X^2\in \mathcal X^2$ is defined as
\begin{align*}
\hsic(X^1,X^2)\coloneqq \vert \vert \mu_{\mathbb P_{X^1X^2}}-\mu_{\mathbb P_{X^1}\times \mathbb P_{X^2}} \vert \vert^2_{H_k}
\end{align*}
where $\mathbb P_{X^1}\times \mathbb P_{X^2}$ denotes the product measure of $\mathbb P_{X^1}$ and $\mathbb P_{X^2}$.
\end{definition}
\noindent
This was proposed in \cite{kernelteststatisticalindependence}. In \cite{equivalencedistancerkhs} it was shown to be equivalent, under a certain choice of kernel, to a statistic earlier proposed by \cite{distancecovariance}, called distance-covariance. We will mainly prove statements for dHSIC, which will then carry over to HSIC. One thing to note is that to ensure that $\hsic(X^1,X^2)=0$ if and only if $X^1 \independent X^2$, it suffices for $k^1$ and $k^2$ to be $c_0$-universal, but this is not required. In \cite{gretton2015simpler} it is shown that it suffices for $k^1$ and $k^2$ each to be characteristic, for example, which is implied by each being $c_0$-universal. 

\section{Two permutation tests}
\subsection{Notation} \label{sec:notation}
We follow the notation of \cite{pfister2018kernel} for the permutation test of dHSIC. Define maps 
$\psi^i:\{1,\dots,n\}\to \{1,\dots,n \}$ for $i=1,\dots,d-1$, and ${\psi=(\psi^1,\dots,\psi^{d-1})}$.
Then $\psi$ maps $\mathcal X^n$ to itself by
\begin{align*}
\psi D\coloneqq \psi(x_1,\dots,x_n)\coloneqq \left(x^{\psi}_{n,1},\dots,x^{\psi}_{n,n} \right) ,
\end{align*}
where
\begin{align*}
x^{\psi}_{n,i}\coloneqq \left( x^1_{i},x^2_{\psi^1(i)},\dots,x^d_{\psi^{d-1}(i)}\right).
\end{align*}
For our purposes $\psi^i$ will all be permutations of $\{1,\dots,n\}$. Note that we keep the first coordinate fixed, and permute the remaining $d-1$ coordinates. Hence there are $(n!)^{d-1}$ different vectors of such permutations in total.\\

Permutation tests compare $\widehat{ \text{dHSIC}}(D)$ with the statistic recomputed on permuted datasets, i.e. with $\widehat{ \text{dHSIC}}(\psi_{i}D)$ for $i=2,\dots,B+1$ for some $B$ (the indexing becomes apparent in the next section). In particular, if we arrange the $B+1$ elements (the original statistic and the $B$ `permuted' statistics) as a vector, we study the rank of the original statistic, where the rank of the largest element is taken to
be 1. When there are ties, for simplicity we consider two ways of dealing with it. 

\subsubsection{Breaking ties at random}

Say a vector $v\in \mathbb R^n$ has $k$ repeated elements that all have the value $a$. Furthermore let there be $s$ elements strictly smaller than $a$ and $l$ elements strictly larger than $a$, so that $s+k+l=n$. When we say we break the ties at random to compute the rank of each element, we mean that the rank of an element with value $a$ is distributed uniformly on $s+1,s+2,...,s+k$. 

\subsubsection{Breaking ties conservatively}

Breaking ties at random may not always be desirable and one may also break ties conservatively. Say we have permutation vectors $\psi_i$ for $i=2,\dots,B+1$. Then we can also define $R$ as
\begin{align*}
R=1+\sum_{i=2}^{B+1} 1\{\widehat{ \text{dHSIC}}(\psi_{i}D) \geq \hsic(D) \}.
\end{align*}
\textit{When we do not mention otherwise, we will break permutations conservatively.} In practice it seems plausible that observing ties in statistics is rare when the random variables involved are continuous. 

\subsection{Defining two permutation tests}

We are now ready to define two testing procedures, a permutation test enumerating all permutations and a permutation test sampling a fixed number of independent random permutations, uniformly from the symmetric group. 

\begin{definition}(Permutation test dHSIC enumerating all permutations) Let ${\psi_i=(\psi_i^1\dots,\psi_i^{d-1})}$ for $i=2,\dots,(n!)^{d-1}$ be all vectors of permutations such that at least one of the entries of $\psi_i$ is not the identity permutation. Let $\psi_1\coloneqq(\id,\dots,\id)$. Then let $R$ be the rank of the first entry of the vector
\begin{align*}
\left(\widehat{ \text{dHSIC}}(\psi_1D),\dots,\widehat{ \text{dHSIC}}(\psi_{(n!)^{d-1}}D) \right)
\end{align*}
when breaking ties at random. Reject if $ p_D \coloneqq  R/(n!)^{d-1} \leq \alpha$. The quantity $p_D$ denotes the $p$-value of the permutation test enumerating all permutations and we call $\alpha$ the level of the test.
\end{definition}
Again we can also break ties conservatively. If the test breaking ties at random has correct type 1 error rate, so will the test that breaks ties conservatively. 
\begin{definition}(Permutation test dHSIC sampling $B\times (d-1)$ permutations) Let $\psi_i=(\psi_i^1,\dots,\psi_i^{d-1})$ for $i=2,\dots,B+1$ be i.i.d. vectors of $d-1$ permutations sampled i.i.d. uniformly from $S_n$. Let $\psi_1\coloneqq(\id,\dots,\id)$. Then let $R$ be the rank of the first entry of the vector
\begin{align*}
\left(\widehat{ \text{dHSIC}}(\psi_1D),\dots,\widehat{ \text{dHSIC}}(\psi_{B+1}D) \right)
\end{align*}
when breaking ties at random. Reject if $\hat p \coloneqq R/(B+1)\leq \alpha $. The quantity $\hat p$ denotes the $p$-value of the permutation test enumerating a finite sample of permutations and we call $\alpha$ the level of the test.
\end{definition}
We have the following equality for the $p$-value of the permutation test enumerating all permutations, when we break ties at random,
\begin{align*}
p_D=\mathbb P\left(\widehat{ \text{dHSIC}}(\psi D) \geq \widehat{ \text{dHSIC}}(D)\, \bigm| \, D\right )
\end{align*}
where $\psi$ is a random vector of $d-1$ permutations, each of which is chosen uniformly and independently from the permutation
group $S_n$.

The finite-sample permutation test (breaking the ties conservatively) has $p$-value
\begin{align*}
\hat p =\frac{1}{B+1}+\frac1{B+1}\sum_{i=2}^{B+1} 1\{\widehat{ \text{dHSIC}}(\psi_i D ) \geq \widehat{ \text{dHSIC}}(D)\}.
\end{align*}
It is clear that, for each fixed dataset $D$, it holds that $\hat p$ converges to $p_D$ almost surely as $B\to \infty$. In fact, given $D$, it holds that
\begin{align*}
\hat p=\frac{1}{B+1}+\frac{1}{B+1}Z
\end{align*}
where
\begin{align*}
Z|D\sim \text{Binom}(B, p_D).
\end{align*}

\section{Relevant work on consistency and type 1 error rate of permutation tests}\label{sec:relevant_literature}
Correct type 1 error rate of permutation tests, independent of the test statistic used, has been known for a long time: see, for example, \cite{hoeffding1952large}, for the test using all permutations. That the type 1 error rate is correct has also been proved in the context of dHSIC by \cite{pfister2018kernel}, although 
with the additional restriction for the randomly sampled permutation test that the data come from a continuous distribution. 
Although this is not a difficult step, our framework also allows for non-continuous data. We follow the proofs by \cite{berrettinformation} 
that appeared in the context of independence testing via mutual information, presenting these proofs in more detail in section \ref{sec:level}.\\

The issue of consistency of $\text{dHSIC}$ in combination with a permutation test was raised in \cite{pfister2018kernel}. 
That work proves consistency of the bootstrap test, and suggests (in remark 3.2) that consistency of the permutation test could 
be proved following the approach by which \cite{romano1989bootstrap} demonstrated consistency of permutation and bootstrap for 
a wide class of statistics. Specialising the more general work of \cite{romano1989bootstrap} to our setting, let $\tau: P \to P^1 \times \cdots \times P^d$ be the map sending a probability distribution to the product of its marginals. Then \cite{romano1989bootstrap} discusses permutation tests for statistics of the form
\begin{align*}
T_n=\sqrt n \delta_{\mathcal V}(\hat P_n,\tau \hat P_n)
\end{align*}
where 
\begin{align*}
\delta_{\mathcal V}(P,Q)=\sup\{|P(V)-Q(V)|:V\in \mathcal V \}
\end{align*}
for a suitably chosen collection of events $\mathcal V$. As noted in Remark 3.2 of \cite{pfister2018kernel} the statistic $T_n$ resembles dHSIC, but dHSIC is a supremum over functions in an RKHS, rather than a supremum over indicator functions.\\

Generalising from indicators to more general classes of functions is a well worn path, but we observe that the result --- consistency
of the permutation test --- may be derived by a simpler argument, one that was developed already by Hoeffding in the 1950s \cite{hoeffding1952large}. 
Recall that the test using all permutations tells us to reject the null hypothesis if $p_D=R/(n!)^{d-1}\leq \alpha$. Let 
\begin{align*}
t^{1}_n(D) \geq t^2_n(D)\geq  \dots \geq t^{(n!)^{d-1}}_n(D)
\end{align*}
be the ordered permuted statistics and let $t_n(D)\coloneqq \text{dHSIC}(D)$. Finally, let $a = \lfloor \alpha (n!)^{d-1}\rfloor $, 
and note that $a/(n!)^{d-1}\to \alpha$. We reject $H_0$ in particular when $t_n(D)>t^a(D)$. This yields the following lower bound on the power:
\begin{align*}
\mathbb P(p_D\leq \alpha )\geq \mathbb P(t_n(D)>t^a_n(D)).
\end{align*}
Then \cite{hoeffding1952large} poses two conditions for a test: \\
\textbf{Condition A:} There exists a constant $\lambda$ such that $t^a_n(D)\to \lambda$ in probability. \\
\textbf{Condition B}: There exists a function $H(y)$, continuous at $y=\lambda$, such that for every $y$ at which $H(y)$ is continuous it holds that
\begin{align*}
\mathbb P(t_n(D)\leq t)\to H(y).
\end{align*}
If the distribution of $D$ and the test statistic satisfy these conditions, then it is easy to see that
\begin{align*}
\lim_{n\to \infty} \mathbb P(p_D\leq \alpha )&\geq \lim_{n\to \infty} \mathbb P(t_n(D)>t^a_n(D)) \\& =1-H(\lambda).
\end{align*}
As a result, if we were to show that for any distribution of $X$ in which the null hypothesis is false these two conditions are met for $\text{dHSIC}$,
and that in addition  $H(\lambda)=0$, we may conclude that the test is consistent. This can be done in a rather simple way: first in Section \ref{sec:consistency} we prove that for
$\psi=(\psi^1,\dots,\psi^{d-1})$ a vector of i.i.d. random permutations of $\{1,\dots,n\}$, sampled uniformly from the symmetric group, it holds that
\begin{align*}
\widehat{\text{dHSIC}}(\psi D)\to 0
\end{align*}
in probability. It is easy to see that this implies $t^a_n(D) \to 0$ in probability, proving condition A with $\lambda=0$. Finally, when we use a universal kernel the statistic $t_n(D)=\widehat{\text{dHSIC}}(D) \to \text{dHSIC}(P)>0$ in probability. So ${H(y)=1\{\text{dHSIC}(P)\leq y \}}$ satisfies condition B with $H(\lambda)=H(0)=0$. This is shown in more detail in Section \ref{sec:consistency}.

\section{Correct type 1 error rate of permutation tests \label{sec:level}}
\begin{definition}(Correct type 1 error rate) Let $\phi(X_1\dots,X_n)\in \{0,1\}$ be a hypothesis test of level $\alpha$ returning $1$ if it rejects $H_0$, and $0$ otherwise. The test $\phi$ of level $\alpha$ is said to have correct type 1 error rate if for any distribution $\mathbb P_X \in H_0$, i.e., such that $\mathbb P_X=\mathbb P_{X^1}\times \dots \times \mathbb P_{X^d}$, it holds that
\begin{align*} \mathbb P_{X}(\phi(X_1,\dots,X_n)=1)\leq \alpha.
\end{align*}	
\end{definition}

In our case, to show the two defined permutation tests have correct type 1 error rate, it suffices to show that for any $\mathbb P_X\in H_0$, it holds that $\mathbb P_{X}(p_D\leq \alpha)\leq \alpha$ and $\mathbb P_{X}(\hat p \leq \alpha)\leq \alpha$. Note that in the first probability the random element is $D$, the dataset, and in the second the random elements are both $D$ and the random permutations.\\

We now prove the two permutation tests indeed have correct type 1 error rate. This was done also in \cite{pfister2018kernel}, but for the randomly sampled permutation test the assumption was made the data was continuous. We follow the approach of \cite{berrettinformation} (Lemma 1, Section 3 of \cite{berrettinformation}) that proved correct type 1 error rate of mutual independence testing. By breaking ties at random or conservatively, we do not need to assume the data is continuous.
\begin{theorem}(Correct type 1 error rate when enumerating all permutations) Assume $H_0$ is true, i.e. the $x_1,\dots,x_n \in \mathcal X$ are sampled $i.i.d$ from a distribution $\mathbb P_{X^1}\times \cdots \times \mathbb P_{X^d}$. Then the permutation test enumerating all permutations with level $\alpha$ rejects with probability at most $\alpha$. 
\end{theorem}
\begin{proof}
View the dataset $D=(x_1,\dots,x_n)$ as a random vector in $\mathcal X^n$. Let $\psi_1=(\id,\dots,\id)$. Then under the null hypothesis $\psi_jD\stackrel{d}{=}D=\psi_1D$ for all $j$. Now there are $(n!)^{d-1}-1$ vectors of permutations whose components are not all the identity permutation. Let $\psi_2,\dots,\psi_{(n!)^{d-1}}$ be those vectors in random order, such that each ordering is equally likely.

We claim that in this case the vector 

\begin{align*}
\left(\psi_1D,\dots,\psi_{(n!)^{d-1}}D\right)
\end{align*} 
is exchangeable. That is, we claim  
\begin{align*}
& \left(\psi_1D,\dots,\psi_{(n!)^{d-1}}D\right) \\ \stackrel{d}{=}  & \left(\psi_{\sigma(1)}D,\dots,\psi_{\sigma((n!)^{d-1})}D\right)
\end{align*}
for any permutation $\sigma$ of $\{1,\dots,(n!)^{d-1}\}$. Indeed, by the remark above, the first entries of the two vectors are equal in distribution. It is not hard to see that all the remaining entries are the $(n!)^{d-1}-1$ permutation vectors whose components do not all equal the identity permutation, and each order is equally likely. Consequently

\begin{align*}
\left(\widehat{ \text{dHSIC}}(\psi_1D),\dots,\widehat{ \text{dHSIC}}(\psi_{(n!)^{d-1}}D) \right)
\end{align*}
is exchangeable too. Breaking ties at random, each entry is equally likely to have any given rank, and in particular the rank of the first (and every other) entry is uniformly distributed on $\{1,\dots,(n!)^{d-1}\}$. 
\end{proof}
\begin{theorem}(Correct type 1 error rate when using a finite sample of permutations) Assume $H_0$ is true, i.e. the $x_1,\dots,x_n \in \mathcal X$ are sampled $i.i.d$ from a distribution $\mathbb P_{X^1}\times \cdots \times \mathbb P_{X^d}$. Then the permutation test sampling a finite number of permutations with level $\alpha$ rejects with probability at most $\alpha$. 
\end{theorem}
\begin{proof}
The proof is nearly identical to the proof above, except that now we notice that
\begin{align*}
\left(\psi_1D,\dots,\psi_{B+1}D\right)
\end{align*}
is an exchangeable vector.
\end{proof}
Note that these proofs do not use any property of dHSIC. The same proofs would work for any function of the dataset.

\section{Consistency of permutation test with test statistic dHSIC \label{sec:consistency}}
As in the previous section let $\phi(X_1,\dots,X_n)$ equal 1 if the null hypothesis is rejected and 0 otherwise. \begin{definition}(Consistency) The test $\phi$ is called consistent if for every distribution $\mathbb P_{X}\in H_1$ such that $\mathbb P_{X^1,...,X^d}\not =\mathbb P_{X^1}\times \cdots \mathbb \times \mathbb P_{X^d}$, it holds that\\
 
\begin{align*}
\text{(Consistency:)} \qquad \lim_{n\to \infty} \mathbb P_{X}(\phi(X_1,\dots,X_n)=1)=1.
\end{align*}
\end{definition}
That is, a test is consistent if for every fixed alternative hypothesis, the rejection rate converges to $1$ as the sample size grows to infinity. To prove this is the case for our proposed tests we make one assumption, that is satisfied by the Gaussian kernel and for any bounded kernel.\\

\textbf{Assumption 1:} Assume Setting $1$ with $n\geq 2d$, then we furthermore assume that for every $(i_1,\dots,i_{2d})\in\{1,\dots,n \}^{2d}= M_{2d}(n)$\begin{align*}
\mathbb E \bigg\lvert \prod_{j=1}^d k^j(X^j_{i_{2j-1}}X^j_{i_{2j}}) \bigg \rvert<C
\end{align*}
for the same constant $C>0$. 

We begin by proving the empirical dHSIC of a randomly permuted sample converges to zero in probability.
\begin{theorem}\label{thm:permutedhsic}
Let $\psi=(\psi^1,\dots,\psi^{d-1})$ be a vector of $i.i.d.$ random permutations. Then 
\begin{align*}
    \widehat{ \text{dHSIC}}(\psi D)\to 0
\end{align*}
in probability. 
\end{theorem}
\begin{proof}
We note that since $\widehat{\text{dHSIC}}$ is nonnegative it suffices to show that 
\begin{align*}
\lim_{n\to\infty}\mathbb E \bigl[ \widehat{ \text{dHSIC}}(\psi D) \bigr]=0.
\end{align*}
For the context of this proof only, when $\psi=(\psi^1,\psi^2,\dots,\psi^d)$ is a $d$-tuple of permutations on $n$ symbols
with $\psi^1=\id$ we redefine $\psi D$ to be $\psi' D$ as defined in section  \ref{sec:notation}, where $\psi'=(\psi^2,\dots,\psi^d)$. With this notation the 
permuted $\widehat{ \text{dHSIC}}$ statistic may be written
\begin{align*}
\widehat{ \text{dHSIC}}(\psi D) = & \frac{1}{n^2} \sum_{M_2(n)} \prod_{j=1}^dk^j(X_{\psi^j(i_1)}^j,X^j_{\psi^j(i_2)})+\frac{1}{n^{2d}}\sum_{M_{2d}(n)}\prod_{j=1}^dk^j(X^j_{\psi^j(i_{2j-1})},X^j_{\psi^j(i_{2j})})\\
&-\frac{2}{n^{d+1}}\sum_{M_{d+1}(n)}\prod_{j=1}^dk^j(X^j_{\psi^j(i_1)},X^j_{\psi^j(i_{j+1})}),
\end{align*}
which we abbreviate as $A_n+B_n-2C_n$. We now aim to show that $\lim_n\mathbb EA_n=\lim_n\mathbb EB_n=\lim_n\mathbb EC_n=\zeta$ where
\begin{align*}
\zeta=\prod_{j=1}^d \mathbb E\left(k^j(X^j,\tilde X^j) \right),
\end{align*}
where $X^j$ and $\tilde X^j$ are independent copies of the same random variable. The main observation to make in this proof is that, as the sample size grows to infinity, almost all of the terms in the sum will have $2d$ distinct indices. More specifically, write
\begin{align*}
A_n =&\frac{1}{n^2} \sum_{U(2,\psi,n)}\prod_{j=1}^dk^j(X_{\psi^j(i_1)}^j,X^j_{\psi^j(i_2)}) \\ & +\frac{1}{n^2} \sum_{R(2,\psi,n)}\prod_{j=1}^dk^j(X_{\psi^j(i_1)}^j,X^j_{\psi^j(i_2)}) 
\end{align*}
where 
\begin{align*}
U(2,\psi,n)\coloneqq&\Big\{ (i_1,i_2)\in M_2(n):{}\\
&\qquad (\psi^j(i_1),\psi^j(i_2) : j=1,\dots,d) \text{ are } 2\times d \text{ distinct elements} \Big\},
\end{align*}
\begin{align*}
R(2,\psi,n)\coloneqq M_2(n)\setminus U(2,\psi,n).
\end{align*}
Conditioning on $\psi$ and using the tower property we find
\begin{align*}
\mathbb E\left( A_n \right) &= \mathbb E\left( \mathbb E\left( A_n \vert \psi\right) \right) \\
&= \mathbb E\left( \frac{1}{n^2} \sum_{U(2,\psi,n)}\mathbb E\left(\prod_{j=1}^dk^j(X_{\psi^j(i_1)}^j,X^j_{\psi^j(i_2)})\Bigg\vert \psi \right) \right)\\ &\quad+  \mathbb E\left( \frac{1}{n^2} \sum_{R(2,\psi,n)} \mathbb E\left( \prod_{j=1}^dk^j(X_{\psi^j(i_1)}^j,X^j_{\psi^j(i_2)}) \Bigg\vert \psi \right)\right)\\
&= \mathbb E\left( \frac{\lvert U(2,\psi,n) \rvert}{n^2} \right) \prod_{j=1}^d \mathbb E\left(k^j(X^j,\tilde X^j) \right) \\ &\quad + \mathbb E\left( \frac{\lvert R(2,\psi,n) \rvert}{n^2} \right) \mathcal O(1) \\ &\to \zeta.
\end{align*}
Note that in the last equality, we use that in the first sum all indices in the product are distinct and the expectation factorizes, and in the second sum the estimate $\mathcal O(1)$ follows from Assumption 1, where we assumed all expectations of the given form are bounded by some constant. The limit then follows from the fact that \begin{align*}
\mathbb E \vert U(2,\psi,n)\vert/n^2 &= \frac{ n(n-1)}{n^2}\cdot\\ &\qquad \mathbb P\left((\psi^j(1),\psi^j(2) : j=1,\dots,d) \text{ are } 2\times d \text{ distinct elements}  \right) \\ &=
\frac{n(n-1)}{n^2}\frac{ {n -2 \choose 2} {n -4 \choose 2}  \cdots {n -2d+2 \choose 2}  }{{n  \choose 2} {n  \choose 2} \cdots {n  \choose 2} } \\ & \to 1.
\end{align*}
As a result 
\begin{align*}
\lim_{n\to\infty}\mathbb E(A_n) =\prod_{j=1}^d \mathbb E\left(k^j(X^j,\tilde X^j) \right) = \zeta.
\end{align*}
This argument can be repeated for $B_n$ and $C_n$. Specifically, observe $B_n$ and $C_n$ are sums over indices in $M_{2d}(n)$ and $M_{d+1}(n)$ respectively, where the summands take the form $\prod_{j=1}^dk^j(X^j_{i_j}, X^j_{\tilde i_j})$ for some multi-indices $I\coloneqq (i_1,\tilde i_1,\dots,i_d,\tilde i_{d})$. We then split these 
into sums over $I$ consisting of $2d$ distinct integers, and those over multi-indices with repeated components. 
Lastly, we remark that the multi-indices $I$ result from randomly permuting the data, and consequently the probability these $2d$ numbers are all distinct converges to 1
as the sample size goes to infinity. Conditioned on the event that indeed all $2d$ integers are distinct the expectation of the product is $\zeta$.
\end{proof}

We now prove that $t^a_n(D)$ (see Section \ref{sec:relevant_literature}) converges to zero in probability.   
\begin{theorem}(Convergence of $t_n^a(D)$)\label{thm:convergencequantile}
Let $\mathbb P_X$ be any distribution such that $\mathbb P_{X}\neq \mathbb P_{X^1}\times \dots \times \mathbb P_{X^d}$. Perform a permutation test on $D=(X_1,\dots,X_n)$ using all $(n!)^{d-1}$ permutation vectors. Let $t^1_n(D)\geq \dots \geq t^{(n!)^{d-1}}$ be the values of $\text{dHSIC}$ computed on all permutations of the data. Let $a=\lfloor \alpha (n!)^{d-1}\rfloor$ for $\alpha \in (0,1)$. Then $t^a_n(D)\to 0$ in probability. 
\end{theorem}
\begin{proof}
Let $\psi$ be a permutation vector consisting of $d$ i.i.d uniformly chosen permutations. Note that $t^a_n(D)\geq\epsilon$ implies that \begin{align*}
\mathbb P( \text{dHSIC}(\psi D) \geq \epsilon \, | \, D ) \geq a/(n!)^{d-1} .
\end{align*}
Consequently, for $\epsilon>0$, using Markov's inequality in the second estimate,
\begin{align*}
\lim_{n\to \infty}\mathbb P(t_n^a(D) \geq \epsilon )  &\leq \lim_{n\to \infty} \mathbb P \left[ \mathbb P( \text{dHSIC}(\psi D) \geq \epsilon \, \bigm| \, D ) \geq a/(n!)^{d-1} \right] \\ 
& \leq \lim_{n \to \infty} \frac{\mathbb E\left[\mathbb P( \text{dHSIC}(\psi D) \geq \epsilon \, \bigm| \, D )\right]}{a/(n!)^{d-1}}\\ 
& = \lim_{n \to \infty} \frac{ \mathbb P(\text{dHSIC}(\psi D) \geq \epsilon )}{ a/(n!)^{d-1}} \\ 
&=0.
\end{align*}
In the last line we use the fact that denominator converges to $\alpha$, and that the numerator converges to 0 by Theorem \ref{thm:permutedhsic}. \end{proof}

We are now ready to prove the test using all permutations is consistent. In fact we prove that $p_D \to 0$ in probability. 

\begin{theorem}
(Convergence of $p_D$, consistency of permutation test using all permutations)\label{thm:consistencyallpermutations}
Let $\mathbb P_X$ be any distribution such that $\mathbb P_{X}\neq \mathbb P_{X^1}\times \dots \times \mathbb P_{X^d}$. Perform a permutation test on $D=(X_1,\dots,X_n)$ using $(n!)^{d-1}$ permutation vectors with test statistic $\text{dHSIC}$ using a characteristic kernel $k$. Let $p_D$ be the resulting $p$-value. Then
\begin{align}
    p_D\to 0
\end{align}
in probability and, in particular, the test is consistent.  
\end{theorem}
\begin{proof}
Note that if $t_n(D)>t^a_n(D)$, then we reject the null hypothesis. Thus for every $\alpha>0$,
\begin{align*}
         \lim _{n\to \infty} \mathbb P(p_D\leq \alpha) &\leq \lim_{n\to \infty} \mathbb P( t_n(D) > t^a_n(D)) \\ & = 1.
\end{align*}
Where we use $t_n(D)=\widehat{\text{dHSIC}}(D) \to \text{dHSIC}(P)>0$ in probability for a characteristic kernel and that by Theorem \ref{thm:convergencequantile}, it holds that $t_n^a(D)\to 0$ in probability. 
\end{proof}
Finally, it is now easy to see that the finite sample permutation test is consistent too.

\begin{theorem}(Consistency using a finite sample of permutations)\label{thm:consistencyfinitenumber}
Let $\mathbb P_X$ be any distribution such that $\mathbb P_{X}\neq \mathbb P_{X^1}\times \dots \times \mathbb P_{X^d}$. Perform a permutation test on a sample of size $n$ using $B$ random permutation vectors $\psi_1,\dots,\psi_B$ of length $d-1$, where $B\geq  \frac{1}{\alpha}-1$ for $\alpha \in (0,1)$,
and suppose the kernel $k$ on $\mathcal X \times \mathcal X$ is characteristic. Then 
\begin{align*}
\lim_{n\to \infty}\mathbb P(\hat p \leq \alpha )=1.
\end{align*}
\end{theorem}

\begin{proof}
By Theorem \ref{thm:consistencyallpermutations}, $p_D\to 0$ in probability. Recall that $\hat p=(1+Z)/(B+1)$ for $Z\sim \text{Binom}(B,p_D)$ and that if $Z=0$, then $\hat p=1/B+1\leq \alpha$ and the test is rejected. So choose $N_1(\epsilon_1,\epsilon_2)$ so large that for $n\geq N_1$, it holds that $\mathbb P(p_D\leq \epsilon_1 )\geq 1-\epsilon_2$, for some $\epsilon_1,\epsilon_2>0$. Then, for $n\geq N_1$,
\begin{align*}
\mathbb P(\hat p \leq \alpha)&\geq \mathbb P(\hat p =1/(B+1) \, \bigm| \, p_D\leq \epsilon_1)\mathbb P(p_D\leq \epsilon_1) \\
&\geq (1-\epsilon_1)^B(1-\epsilon_2) 
\end{align*}
By choosing $\epsilon_1$ and $\epsilon_2$ small enough this can be made arbitrarily close to 1.
\end{proof}

\section{How many permutations to use? \label{sec:howmany}}
Recall that for the finite sample permutation test we sample $B$ permutation vectors, each consisting of $d-1$-permutations. We previously showed type 1 error is correct for any $B$ and the test is consistent for every $B$ such that $1/(B+1)\leq \alpha$. So, if the tests work well for any such value of $B$, how does one decide which value of $B$ to use? There is no definite answer, but we suggest here some relevant considerations.

\subsection{Rejection probabilities for a fixed dataset}
Each fixed dataset has an associated $p$-value based on enumerating all permutations, given by
\begin{align*}
p_D\coloneqq \mathbb P(\widehat{\text{dHSIC}}(\psi D)\geq \widehat{\text{dHSIC}}(D) \, \bigm| \, D)
\end{align*}
Based on the discussion above, the permutation test enumerating all permutations rejects the null hypothesis 
with probability 1 $H_0$ if $p_D\leq \alpha$ and with probability 0 otherwise. 

On the other hand, the finite-sample permutation test (breaking ties conservatively) rejects the null hypothesis if and only if
\begin{align*}
\hat p =\frac{Z+1}{B+1} \leq \alpha
\end{align*}
where
\begin{align*}
Z|D\sim \text{Binom}(B,p_D).
\end{align*}
For a given dataset and associated value $p_D$, we can compute the probability the finite-sample permutation test rejects the null hypothesis as
\begin{align*}
\mathbb P(\hat p \leq \alpha\vert D)=\mathbb P(Z \leq \alpha(B+1)-1 \, \bigm| \, D).
\end{align*}
This is the probability one rejects $H_0$ for a dataset with a parameter $p_D$. This probability is plotted in Figure \ref{plot:rejection_rate_vs_pd}. \\

The effect of the number of permutations on the probability of rejecting $H_0$ can be understood in three regimes of $p_D$, as is seen in Figure \ref{plot:rejection_rate_vs_pd}.
\begin{enumerate}
\item The data is such that $p_D <\alpha$:
In this case, increasing the number of permutations increases the probability of rejecting the null hypothesis for this dataset.
\item The data is such that $p_D = \alpha$:
In this case the probability of rejecting the null hypothesis is approximately $1/2$, since the mean and median of a binomial distribution are very close. 
\item The data is such that $p_D > \alpha$:
In this case, as the number of permutations increases and $\hat p$ gets closer and closer to $p_D$, we reject the null hypothesis more often. So, using fewer permutations actually raises the probability of rejecting the null hypothesis based on this dataset.\\
\end{enumerate}
In summary, using more permutations increases the probability of rejecting the null hypothesis when your dataset has parameter $p_D\leq \alpha$ and lowers the probability of rejecting datasets with parameter $p_D>\alpha$. We stress that the test has correct type 1 error rate regardless of the number of permutations. 

\begin{figure}[H]
\centering

\begin{subfigure}{0.8\textwidth}
  \centering
  \includegraphics[width=1.1\linewidth]{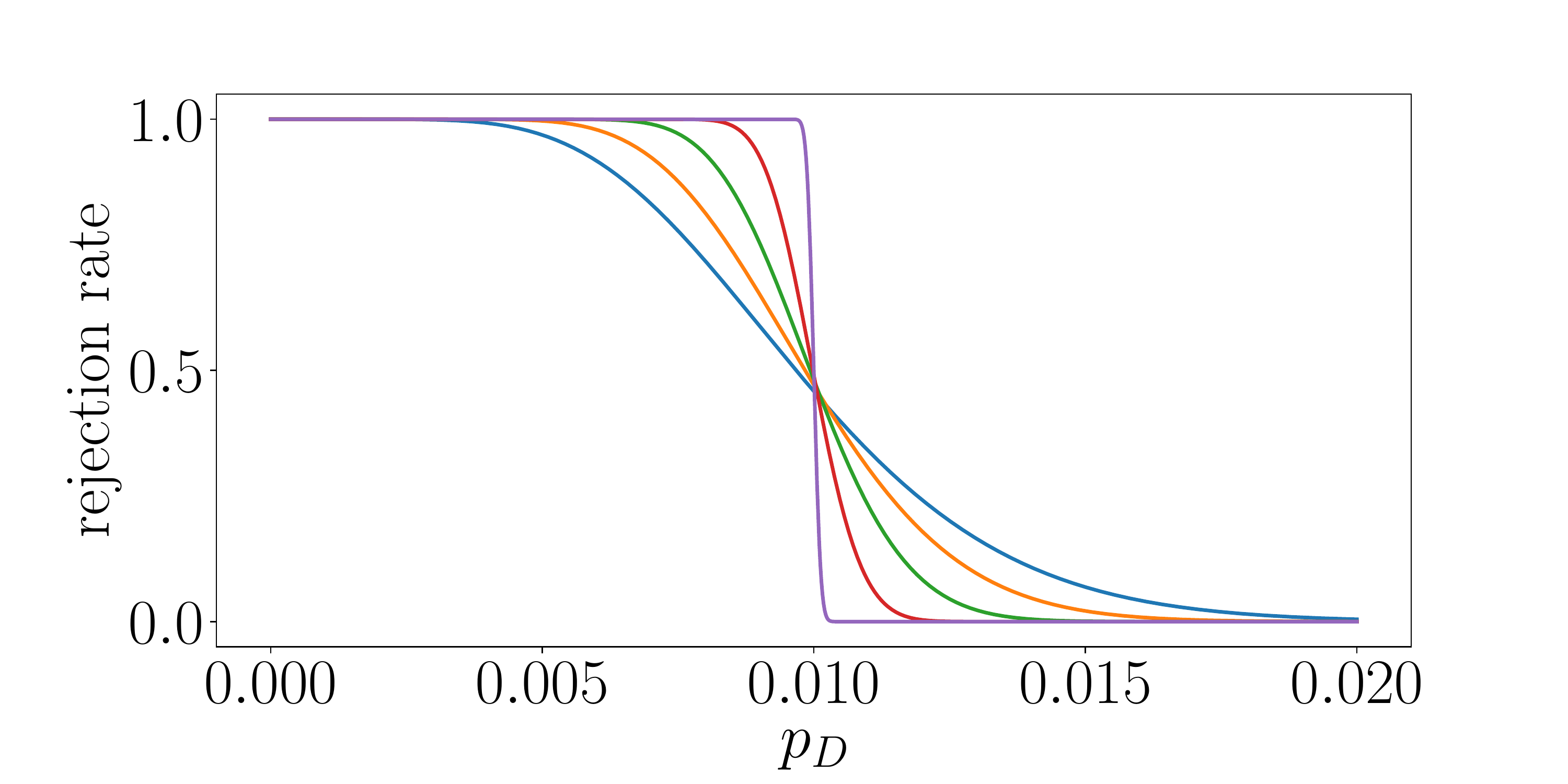}
\end{subfigure}
\begin{subfigure}{0.8\textwidth}
  \centering
  \includegraphics[width=1.1\linewidth]{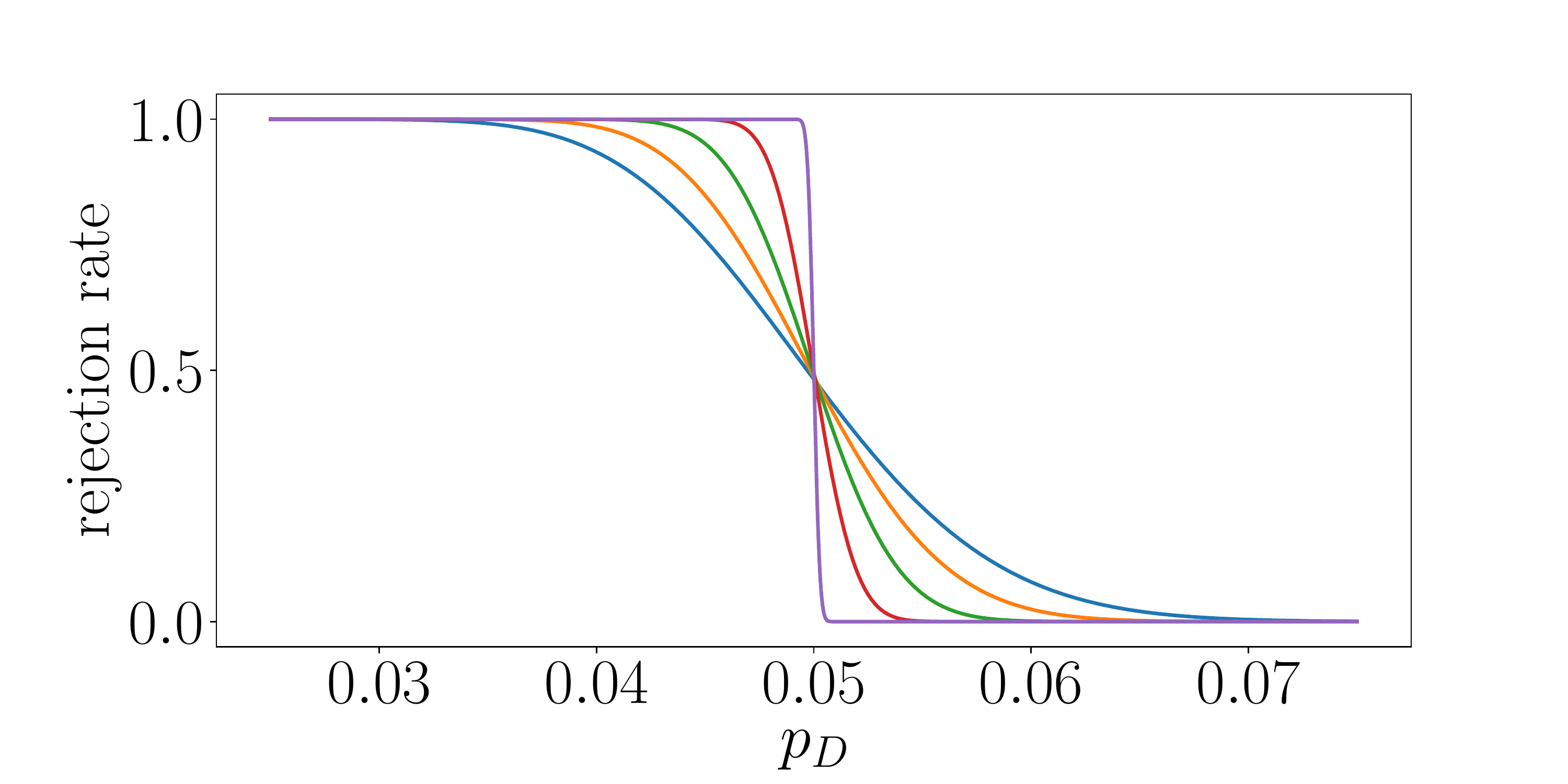}
\end{subfigure}
	
\begin{subfigure}{.9\textwidth}
  \centering
  \includegraphics[width=1.1\linewidth]{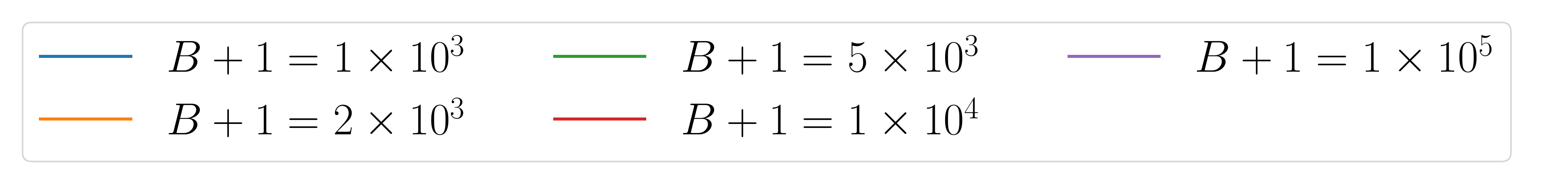}
\end{subfigure}
\caption{\label{plot:rejection_rate_vs_pd} Say you are given (fixed) data $D$ with associated $p_D$. When you enumerate all permutations, you reject if and only if $q\leq \alpha$. These plots plot $\mathbb P(\hat p \leq \alpha \vert p_D)$ for different values of $B$.   }
\end{figure}

\subsection{The effect of the number of permutations on the power}

The parameter $p_D$ itself is a random quantity too (as it is a function of the data). Say for the sake of simplicity it has a density $f_{p_D}(p)$ on $[0,1]$. The total probability of rejecting the null hypothesis when using a permutation test with $B$ permutation vectors is
\begin{align*}
\mathbb P(\hat p\leq \alpha)=\int_{[0,1]}\mathbb P(Z \leq \alpha(B+1)-1|p_D=p)f_{p_D}(p)dp.
\end{align*}
While we plotted the function $\mathbb P(Z \leq \alpha(B+1)-1|p_D=p)$ in Figure \ref{plot:rejection_rate_vs_pd}, the quantity $f_{p_D}(p)$ will depend very much on the data generating mechanism. It will be approximately uniform under the null hypothesis, but analytic descriptions of $f_{p_D}$ are complicated for arbitrary distributions of $X$. We perform two simulation studies to illustrate the relationship between power and $B$.
We study the case where $d=2$.\\

\textbf{Scenario 1:} Let 
\begin{align*}
X^2=\theta X^1+ \epsilon
\end{align*}
where $X^1,\epsilon\sim \mathcal N(0,I_5)$ independently where $I_5$ is the $5-$dimensional identity matrix. When $n=100$ we find that the power is nearly identical for all numbers of permutations and for all values of $\theta$, as shown in Figure \ref{plot:power_vs_theta}. An explanation is that the variance in the underlying $p$-value $p_D$ is large when the sample size is $100$, and as a result $f_{p_D}$ has a wide support, and the integral of the functions plotted in Figure 1 with density $f_{p_D}$ all result in the same value. \\

\textbf{Scenario 2:} The previous scenario showed no difference in power between the methods. The explanation was the variance of $p_D$, or the width of the distribution $f_{p_D}$. It is not easy to find distributions of $X^1$ and $X^2$ such that $f_{p_D}$ has a support only in the region where the curves in Figure 1 are separated (so near $\alpha$) and this support is furthermore not symmetric around $\alpha$. So in Scenario 2, for each value of $\theta$, we choose a fixed (nonrandom) dataset. Namely:
\begin{align*}
X^2=\sin(\theta X^1)
\end{align*}
where $X^1_i=i2\pi/100$ for $i=1,\dots,100$ (so again $n=100$). In this case, as $\theta$ increases, the frequency of the oscillation increases and the sample looks less dependent. Indeed in Figure \ref{plot:power_vs_theta} we see that slow oscillations are more often rejected when using more permutations, and fast oscillations are more often rejected when using fewer permutations. 

\begin{figure}[H]
\centering

\begin{subfigure}{1\textwidth}
  \centering
  \includegraphics[width=0.8\linewidth]{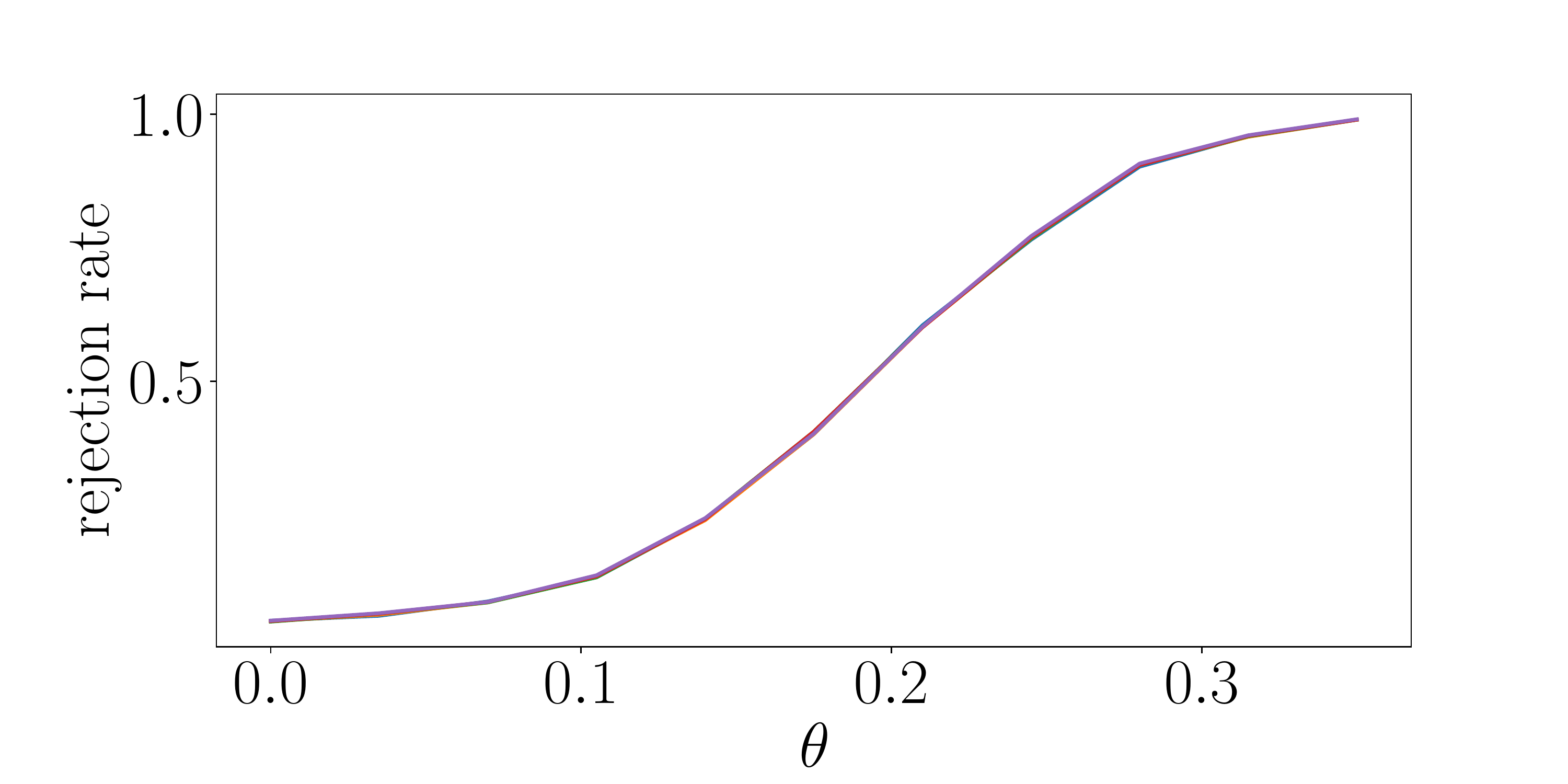}
  \subcaption{The rejection rates of the permutation test with different values of $B$ in Scenario 1. All tests have equal power throughout and all curves overlap.}
\end{subfigure}
\begin{subfigure}{1\textwidth}
  \centering
  \includegraphics[width=0.8\linewidth]{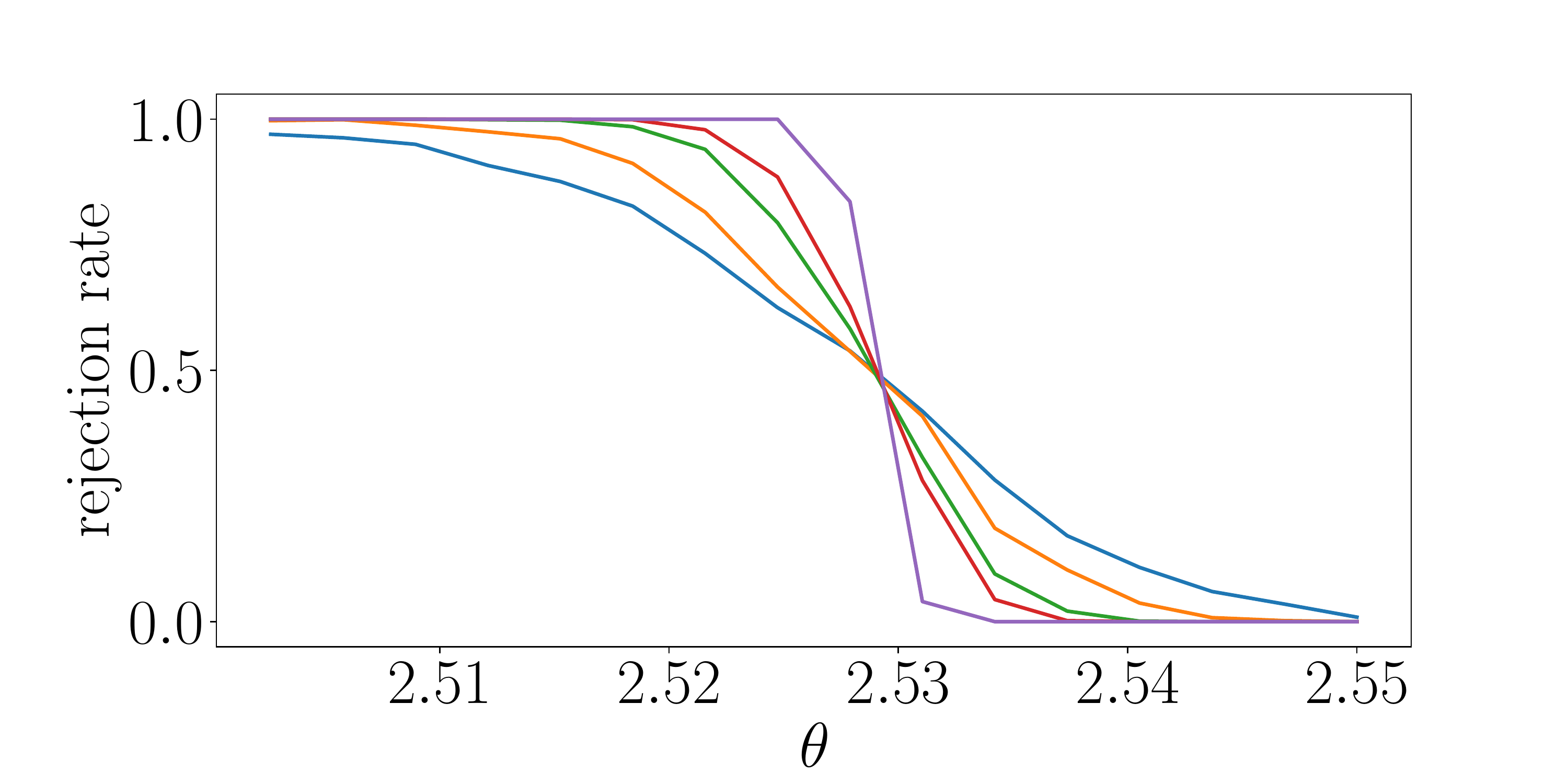}
  \subcaption{The rejection rates of the permutation test with different values of $B$ in Scenario 2. Using more permutations makes one virtually always reject datasets with lower $\theta$ and virtually never with higher $\theta$. Lowering the number of permutation test makes the decision more uncertain.}
\end{subfigure}
\begin{subfigure}{.9\textwidth}
  \centering
  \includegraphics[width=1.1\linewidth]{plots/second_legend.pdf}
\end{subfigure}
\caption{\label{plot:power_vs_theta} The power of the test with different numbers of permutations in Scenario 1 and 2.}
\end{figure}

\subsection{Confidence intervals}

In practice one is not only interested in accepting or rejecting the null hypothesis, but one wants to find a reliable estimate of $p_D$. So we recommend to choose $B$ so large that $\hat p$ is likely to be close to $p_D$.  We recall that, given the dataset,
\begin{align*}
    \hat p =\frac{1}{B+1}+\frac{1}{B+1}Z
\end{align*}
where $Z\sim \text{Binom}(B,p_D)$. Hence accuracy of $\hat p$ can be described simply through confidence intervals of the binomial distribution. That is,  

\begin{align*}
 \mathbb P&\left( \hat p - \epsilon \leq p_D \leq \hat p + \epsilon \right) \\
&=\mathbb P \left( (p_D-\epsilon)(B+1)-1 \leq Z \leq (p_D+\epsilon)(B+1)-1 \right). 
\end{align*}
For a given, $\epsilon, p_D$ and confidence level $1-\lambda$, we can find $B$ such that $\hat p \pm \epsilon $ is a confidence interval of level $1-\lambda$. As we do not know $p_D$, we could choose $B$ so large that for \textit{any} $p_D$, the interval $\hat p \pm \epsilon$ is a $1-\lambda$ confidence interval. \\

However, $B$ will always be highest for $p_D \approx 0.5$ as that value maximizes the variance of the binomial distribution. As the accuracy of the estimate is more important when $p_D$ is close to $\alpha$, we may want to reduce the number of permutation vectors needed by dividing the data in two possible cases:

Case 1: The data is such that $p_D\in [0,C]$ for some $C\in (\alpha,1)$. In this case choose $B$ so large that $\hat p \pm \epsilon$ is a $1-\lambda$ confidence interval for some $\epsilon$. Case 2: The data is such that $p_D\in (C,1]$. In that case, we simply check if the maximum width of the $1-\lambda$ confidence interval matches our desired accuracy - in particular we check if $p_D>C$ implies that $\hat p$ is very unlikely to be near $\alpha$. Using these two cases we we allow for more error when $p_D\geq C$. Say for example $\alpha=0.05$, $C=0.10$, $\epsilon=0.005$, and $1-\lambda=99\%$. Then we need $B=2.3\times 10^4$ permutations to ensure $\hat p\pm \epsilon$ is a $99\%$ confidence interval whenever $p_D\leq C$. If $p_D\geq C$, then the widest width of a $99\%$ confidence interval is $0.01$, so our estimated $p$-value $\hat p$ is still accurate. \\

\section{Conclusion}
We have studied kernel measures of dependence and how they are combined with permutation tests to perform hypothesis testing. Our main contribution is proving the consistency of the permutation test with statistic dHSIC with a universal kernel. This implies in particular consistency of the permutation test with test statistic HSIC. Additionally we show that for each number of permutations and for each number of samples the probability of making a type 1 error is at most $\alpha$. This last statement was a known result, and we proved it following the method used by \cite{berrettinformation} in the context of independence testing by mutual information, extending it to testing mutual independence. We further gave examples of how one may go about choosing a number of permutations in practice.  
\bibliographystyle{plain}
\bibliography{refs}

\end{document}